\documentclass{amsart}
\usepackage[utf8]{inputenc}
\usepackage{amsfonts, amsmath, amsthm, amssymb,tikz-cd}
\usepackage{graphicx,color}

\title[Some geometric properties of matrix means]{SOME GEOMETRIC PROPERTIES OF MATRIX MEANS WITH RESPECT TO DIFFERENT DISTANCE FUNCTIONS}
\author[T. H. Dinh]{Trung Hoa Dinh}
\address{Department of Mathematics  \\ Troy University \\ Troy, Alabama 36082, USA}
\email{thdinh@troy.edu}
\author[R. Dumitru]{Raluca Dumitru}
\address{Department of Mathematics and Statistics \\ University of North Florida  \\ Jacksonville, FL 32224}
\email{raluca.dumitru@unf.edu}
\author[J. A. Franco]{Jose A. Franco}
\address{Department of Mathematics and Statistics \\ University of North Florida  \\ Jacksonville, FL 32224}
\email{jose.franco@unf.edu}


\theoremstyle{plain}
\newtheorem{theorem}{Theorem}

\newtheorem{corollary}[theorem]{Corollary}
\newtheorem{remark}[theorem]{Remark}

\newtheorem{proposition}[theorem]{Proposition}
\def\Tr{{\rm Tr\,}}
\def\M{{\mathbb{M}\,}}
\theoremstyle{definition}

\begin{document}

\maketitle

\begin{abstract}
In this paper we study the monotonicity, in-betweenness and in-sphere properties of matrix means with respect to Bures-Wasserstein, Hellinger and Log-Determinant metrics. More precisely, we show that the matrix power means (Kubo-Ando and non-Kubo-Ando extensions) satisfy the in-betweenness property in the Hellinger metric. We also show that for two positive definite matrices $A$ and $B$, the curve of weighted Heron means, the geodesic curve of the arithmetic and the geometric mean lie inside the sphere centered at the geometric mean with the radius equal to half of the Log-Determinant distance between $A$ and $B$.
\end{abstract}

\maketitle

\section{Introduction}

Let $\M_n$ be the algebra of $n\times n$ matrices over $\mathbb{C}$ and $\mathcal{D}_n$ denote cone the positive definite elements of $\M_n$. Denote by $I$ the identity matrix of $\M_n$. For a real-valued function $f$ and a Hermitian matrix $A \in \M_n$ the matrix $f(A)$ is understood by means of the functional calculus. The space of density matrices or quantum states is as
$$\mathcal{D}_n^1=\{\rho \in \mathcal{D}_n \ | \ \Tr \rho =1 \}.$$

In \cite{Audenaert2}, Audenaert introduced the concept of ``in-betweenness" and distance monotonicity for matrix means as follows. A matrix mean $\sigma$ is said to satisfy the {\it in-betweenness property} with respect to the metric $d$ if for any pair of positive definite operators $A$ and $B$,
$$d(A, A\sigma B)\le d(A,B).$$

A weighted operator mean $\sigma_t$ is said to satisfy the {\it distance monotonicity} property with respect to the metric $d$ if for any pair of positive definite operators $A$ and $B$, the function $t\mapsto d(A,A\sigma_t B)$ is monotone on $[0, 1]$.

In the same article, he showed that the in-betweenness property is not stronger than the distance monotonicity for the matrix power means as defined by Bhagwat and Subramaninan \cite{bhagwat},
$$\mu_p(t;A,B):=(tA^p+(1-t)B^p)^{1/p}.$$

Using this comparison he showed that the weighted power means satisfy the in-betweenness property when $1\le p\le 2$ and $0\le t \le 1$. He later conjectured that this property should be satisfied for $p\ge 2$. However, in \cite{DDF} we constructed counterexamples for $p=6$. Moreover, we showed that the weighted power means satisfy the in-betweenness for $p=1/2$ and $p=1/4$. Interestingly, in the case of $p=1/2$, the property is satisfied with respect to any metric induced from a unitarily invariant norm, i.e., $d(A,B)=|||A-B|||.$

In \cite{HK} the first author and co-authors introduced the {\it in-sphere property} for matrix means with respect to some distance function $d$ on $\mathcal{D}_n$. A matrix mean $\sigma$ satisfies the in-sphere property with the center $X$ in a metric $d$ if 
$$
d(X, A\sigma B) \le \frac{1}{2}d(A, B).
$$
They showed that the matrix power mean $\mu_p(t;A,B)$ satisfies in-sphere property with respect to the Hilbert-Schmidt norm.

Another kind of the matrix power mean in the sense of Kubo-Ando \cite{Kubo-Ando} is defined as
$$P_p(t, A, B)= A^{1/2}\left (tI + (1-t)(A^{-1/2}BA^{-1/2})^p \right)^{1/p}A^{1/2}.$$
For $p=t=\dfrac{1}{2},$ 
$$
P_{\frac{1}{2}}(1/2, A, B) = 
\frac{1}{4}(A+B+A\sharp B), $$
where $A\sharp B$ is the midpoint of the geodesic curve of weighted geometric mean $A\sharp_t B = A^{1/2}(A^{-1/2}BA^{-1/2})^{t}A^{1/2}$ joining $A$ and $B$. The mean $P_{1/2}(1/2, A, B)$ is the arithmetic mean of the arithmetic mean and the geometric mean of $A$ and $B$, and is called {\it the Heron mean} in the sense of Kubo-Ando. This mean is the main object of the investigations in \cite{Bhatia4, hoa07}. 

Now for positive definite matrices $A$ and $B$, let us define the weighted matrix Heron mean as follows:
$$
H_t(A, B)= t A\sharp B + (1-t)A\nabla B,
$$
where $A\nabla B = (A+B)/2$ is the arithmetic mean of $A$ and $B$. Therefore, $P_{1/2}(1/2, A, B)$ is the intersection point of two curves $P_{1/2}(t, A, B)$ and $H_t(A, B).$ Note that the matrix power mean $P_p(t, A, B)$ joins $A$ and $B$ while the weighted Heron mean $H_t(A, B)$ joins $A\sharp B$ and $A\nabla B.$ 

In recent years there has been considerable interest in the manifold $\mathcal{D}_n$ of positive
definite matrices with the Riemannian metric:
$$
d_R(A, B) = \left( \sum_{i=1}^n \log^2 \lambda_i(A^{-1}B)\right)^{1/2},
$$
where $\lambda_i(A^{-1}B)$ are eigenvalues of the matrix $A^{-1/2} BA^{-1/2}.$

Besides this, there are other distances on $\mathcal{D}_n$ that are important in quantum information theory, signal processing, machine learning and other areas. For example:

\begin{itemize}
    \item Bures-Wasserstein distance in the theory of optimal transport \cite{BhatiaLim}:
    $$
d_b(A, B) = \left(\Tr(A+B) - 2\Tr((A^{1/2}BA^{1/2})^{1/2}) \right)^{1/2}.
$$
\item The Log-Determinant metric in machine learning and and quantum information \cite{sra}:
$$d_l(A, B) = \log \det \frac{A+B}{2}  - 2\log \det(AB).$$ 
\item The Hellinger metric or Bhattacharya metric in quantum information \cite{2016arXiv161103449S}:
$$d_h(A,B)=\big(\Tr(A+B)-2\Tr(A^{1/2}B^{1/2})\big)^{1/2}.$$
\end{itemize}

In this paper, we focus on the study of the monotonicity, in-betweenness and in-sphere properties with respect to Bures-Wasserstein, Hellinger and Log-Determinant metrics. 
The paper is organized as follows. In the next section we study the in-betweenness property of the matrix power means in the Hellinger and Bures-Wasserstein metrics. We show that both $\mu_p(t, A, B)$ and  $P_p(t, A, B)$ satisfy the in-betweenness property in the Hellinger metric $d_h$. 
In addition, we show that the Bures-Wasserstein and Hellinger metrics are equivalent in the cone of positive semidefinite matrices. 
A distance relation between $P_p(t, A, B), \ \mu_p(t, A, B)$ and quantum fidelity is also obtained. In the last section we study the monotonicity and in-sphere property in Log-Determinant metric. We establish a picture which shows that for positive definite matrices $A$ and $B$ three curves $A\sharp_t B$, $H_t(A, B)$ and $(A\sharp B)\sharp_t(A\nabla B)$ lie inside the sphere centered at $A\sharp B$ with the radius $\frac{1}{2}d_l(A, B).$




\section{In-betweenness property of the matrix power means in Bures-Wasserstein and Hellinger metrics}

In quantum information theory, two metrics are of special interest. These are the Bures-Wasserstein and Hellinger metrics. 
In this section,  we study in-betweenness of the matrix means with respect to both these metrics. 

In \cite{Audenaert2}, it was shown that for the power means $\mu_p(t;A,B)$, the ``in-betweenness" property implies distance monotonicity. In the following proposition, we show that this is also true for the Kubo-Ando extension of the power means $P_p(t;A,B)$.

\begin{proposition}\label{Prop4}
Let $A$ and $B$ be positive semidefinite matrices. Then, $d(A,P_p(t;A,B))\le d(A,P_p(s;A,B))$ for $0\le t\le s\le 1$ if and only if $d(A,P_p(r;A,B))\ge d(A,P_p(0;A,B))$ for $r\in[0,1]$ and any metric $d$ on $\mathcal{D}_n$.
\end{proposition}
\begin{proof}
It suffices to show the converse. Assume $0\le s\le t\le 1$ and let $C=A^{-1/2}B A^{-1/2}$. Write $\hat{C}=t I +(1-t)C^p$. Then,
$$P_p(t;A,B)=A^{1/2}\hat{C}^{1/p}A^{1/2}=:\hat{B}.$$
Define $s=r+(1-r)t$ and so,
\begin{align*}
    s I+(1-s)C^{p}& = (r+(1-r)t) I+(1-r)(1-t)C^{p} \\
     & =r I +(1-r)(tI+(1-t)C^{p}) \\
     & =r I +(1-r) \hat{C}.
\end{align*}
Then,
\begin{align*}
    P_p(r, A, \hat{B}) & =A^{1/2}(rI+(1-r)(A^{-1/2}(A^{1/2}\hat{C}^{1/p}A^{1/2})A^{-1/2})^p)^{1/p}A^{1/2}
    \\ &=A^{1/2}(rI+(1-r)\hat{C})^{1/p}A^{1/2}=P_p(s,A,B)
\end{align*}
Moreover, $$P_p(0,A,\hat{B})=\hat{B}.$$
Therefore, $d(A,P_p(t,A,B))\le d(A,P_p(s,A,B))$ reduces to 
$$d(A,P_p(0,A,\hat{B}))=d(A,P_p(t,A,B))\le d(A,P_p(s,A,B))=d(A,P_p(r, A, \hat{B})).$$
\end{proof}

\subsection{Hellinger Metric}

Taking advantage of operator convexity and concavity, on the interval $1/2\le p \le 1$, it is straightforward to show that the power means $\mu_p(t;A,B)$ satisfy the in-betweenness property.
\begin{theorem}
Let $A,B\in \mathcal{D}_n$, $1/2\le p \le 1$ and $0\le t \le 1$. Then,
$$d_h(A,\mu_p(t;A,B))\le d_h(A,B).$$
\end{theorem}
\begin{proof}
This result follows if
$$\Tr(\mu_p(t;A,B)-2A^{1/2}\mu_p(t;A,B)^{1/2})\le \Tr(B-2A^{1/2}B^{1/2}).$$
By the operator convexity of $x\mapsto x^{1/p}$ when $1/2\le p \le 1$,
$$\mu_p(t;A,B)\le t A+ (1-t)B.$$
Thus, the desired result follows if
$$\Tr(t(A-B)-2A^{1/2}\mu_p(t;A,B)^{1/2})\le -2\Tr(A^{1/2}B^{1/2}).$$
By the operator concavity of the map $x\mapsto x^{1/2p}$ when $1/2\le p \le 1$,
$$\mu_p(t;A,B)\ge t A^{1/2}+ (1-t)B^{1/2}.$$
Therefore, the distance monotonicity follows if
$$\Tr(t(A-B)-2(tA+(1-t)A^{1/2}B^{1/2}))\le -2\Tr(A^{1/2}B^{1/2}),$$
or,
$$t\Tr(2A^{1/2}B^{1/2}-(A+B))\le 0$$
which is nothing but the well-known AGM inequality.
\end{proof}

Notice that the mean $\mu_p(t, A, B)$ is a Kubo-Ando mean if and only if $p = \pm 1$. In the following theorem, we show that the Kubo-Ando extension of the power means $P_p(t;A,B)$ satisfies the in-betweenness property. Interestingly, we have counterexamples that show that some well-known Kubo-Ando means do not satisfy this property.  
\begin{theorem}\label{Thm4}
Let $A,B\in \mathcal{D}_n$ and $1/2\le p \le 1$. Then for any $t \in [0, 1]$,
$$d_h(A,P_p(t;A,B))\le d_h(A,B).$$
\end{theorem}
\begin{proof}
By a direct calculation, we obtain that the result follows if
$$\Tr(P_p(t;A,B)-2A^{1/2}P_p(t;A,B)^{1/2})\le \Tr(B-2A^{1/2}B^{1/2}).$$
Note that the inequality is equality at $t=0$. For $t=1$, the left-hand-side becomes $\Tr(-A)$, whereas the inequality becomes the well-known inequality 
$$ \Tr(A+B-2A^{1/2}B^{1/2})\ge 0.$$ Moreover, this implies that the linear interpolation between the end points of left-hand-side are always bounded by right-hand-side, namely, for $t\in (0,1)$,
$$\Tr(B-2A^{1/2}B^{1/2})\ge \Tr(-tA+(1-t)(B-2A^{1/2}B^{1/2})).$$
Thus, it suffices to show that the function $t\mapsto \Tr(P_p(t;A,B)-2A^{1/2}P_p(t;A,B)^{1/2})$ is convex on $(0,1)$. For this purpose, we show that the  function $f(t)= \Tr(P_p(t;A,B))$ is convex and the function  $g(t)=\Tr(A^{1/2}P_p(t;A,B)^{1/2})$ is concave. 

In the following, let $C=A^{-1/2}BA^{-1/2}$. Then we have
\begin{align*}
    f\left(\frac{t+s}{2}\right)& =\Tr\left(A^{1/2}\left(\frac{t+s}{2}I+\left(1-\frac{t+s}{2}\right)C^p\right)^{1/p}A^{1/2}\right)
\\ & =\Tr\left(A^{1/2}\left(\frac{1}{2}(t I+(1-t)C^p)+\frac{1}{2}(s I+(1-s)C^p)\right)^{1/p}A^{1/2}\right)
\\ & \le \Tr\left(A^{1/2}\left(\frac{1}{2}(t I+(1-t)C^p)^{1/p}+\frac{1}{2}(s I+(1-s)C^p)^{1/p}\right)A^{1/2}\right)
\\ & = \frac{1}{2}f(t)+\frac{1}{2}f(s),
\end{align*} where the inequality follows from the operator convexity of the function $x\mapsto x^{1/p}$ when $1/2\le p\le 1$. Therefore, the function $f$ is convex on $(0,1)$. 

Now, let us show that the function $g(t)$ is concave on $(0, 1)$. Note that if the following inequality  
\begin{align}\label{ineq11}
    g\left(\frac{t+s}{2}\right)& =\Tr\left(A^{1/4}\left(A^{1/2}\left(\frac{t+s}{2}I+\left(1-\frac{t+s}{2}\right)C^p\right)^{1/p}A^{1/2}\right)^{1/2}A^{1/4}\right)
\\ \nonumber & \ge \Tr\left(A^{1/2}\left(\frac{t+s}{2}I+\left(1-\frac{t+s}{2}\right)C^p\right)^{1/2p}A^{1/2}\right)
\end{align}
is true, then by the operator concavity of $x\mapsto x^{1/2p}$ for $1/2\le p\le 1$ one can see that 
\begin{align*}
    g\left(\frac{t+s}{2}\right) & \ge \frac{1}{2}\Tr\left(A^{1/2}\left(tI+\left(1-t\right)C^p\right)^{1/2p}A^{1/2}\right)\\ & \quad +\frac{1}{2}\Tr\left(A^{1/2}\left(sI+\left(1-s\right)C^p\right)^{1/2p}A^{1/2}\right)
    \\ & = \frac{1}{2}g(t)+\frac{1}{2}g(s).
\end{align*}
Suffices, now, to establish (\ref{ineq11}) to finish the proof. We show a stronger statement, that is, for $A, H\in \mathcal{D}_n$ and $\alpha \ge 1$,
\begin{equation}\label{inequa11}
\lambda(A H^{1/\alpha}A) \prec_{log}\lambda(A^{1/2}(A^{\alpha/2}H A^{\alpha/2})^{1/\alpha}A^{1/2}).
\end{equation}
Then, for $\alpha =2, H=\left(\frac{t+s}{2}I+\left(1-\frac{t+s}{2}\right)C^p\right)^{1/p}$ and $A$ replaced by $A^{1/2}$ we get (\ref{ineq11}).

First, we show that 
\begin{equation}\label{in11}
\lambda_1(A^{1/2}(A^{\alpha /2}H A^{\alpha /2})^{1/\alpha }A^{1/2})\ge \lambda_1(A H^{1/\alpha }A).
\end{equation}
Once this is established, a standard argument using $k$-antisymmetric tensor powers gives the desired inequality (\ref{inequa11}). Inequality (\ref{in11}) follows if
$$A^{1/2}(A^{\alpha /2}H A^{\alpha /2})^{1/\alpha }A^{1/2}\le I \quad  \implies \quad  A H^{1/\alpha }A\le I,$$
or, equivalently,
$$(A^{\alpha /2}H A^{\alpha /2})^{1/\alpha }\le A^{-1} \quad  \implies \quad H^{1/\alpha }\le A^{-2}.$$
The last implication follows from Furuta's inequality: For $X\ge Y\ge 0$, $q\ge 1$, and $r\ge 0$, 
$$(X^rY^qX^r)^{1/q}\le (X^{q+2r})^{1/q},$$
applied to $X=A^{-1}$, $Y=(A^{\alpha/2}H A^{\alpha /2})^{1/\alpha}$, $q=\alpha $ and $r=\alpha /2$. 
\end{proof}

A straightforward consequence from Theorem \ref{Thm4} and Proposition \ref{Prop4} is the decreasing monotonicity of $d_h(A,P_p(t;A,B))$.
\begin{corollary}
Let $A,B\in \mathcal{D}_n$,  $1/2\le p \le 1$, and  $0\le t \le 1$
Then, the function $t \mapsto d_h(A,P_p(t;A,B))$ is monotonically decreasing.
\end{corollary}

Counterexamples for the fact that the geometric and the harmonic means do not satisfy the in-betweenness property with respect to the Hellinger distance. Consider,
\begin{align*}
    A=\left(
\begin{array}{cc}
 113 & -36 \\
 -36 & 17 \\
\end{array}
\right) \quad \text{ and } \quad  B=\left(
\begin{array}{cc}
 12 & -12 \\
 -12 & 12 \\
\end{array}
\right).
\end{align*}
Then,
$$7.94782=d_h(A, A\sharp B) \not\le d_h(A,B)=7.8729.$$
Similarly, if we consider
$$ A=\left(
\begin{array}{cc}
 58 & -24 \\
 -24 & 10 \\
\end{array}
\right) \quad \text{ and } \quad  B=\left(
\begin{array}{cc}
 13 & -8 \\
 -8 & 5 \\
\end{array}
\right).$$
Then,
$$5.66315=d_h(A, A! B) \not\le d_h(A,B)=4.20652,$$
where $A!B=\mu_{-1}(1/2; A,B)$ is the harmonic mean of $A$ and $B$.

\subsection{Bures-Wasserstein Metric}
Recall that the fidelity between two quantum states $\rho$ and $\sigma$ is defined by
$$\sqrt{\mathcal{F}(\rho,\sigma)}=\Tr((\rho^{1/2}\sigma \rho^{1/2})^{1/2}).$$
And the Bures-Wasserstein metric between two states $\rho$ and $\sigma$ is defined by
$$d_b(\rho,\sigma)=\big(\Tr(\rho+\sigma)-2\sqrt{\mathcal{F}(\rho,\sigma)} \big)^{1/2}.$$ That makes the Bures-Wasserstein metric being important in quantum information theory.

It was proved in \cite[Proposition 9]{2016arXiv161103449S} that the Bures-Wasserstein and Hellinger metrics are equivalent in the space of density matrices. More precisely, they showed that for two quantum states $\rho$ and $\sigma$,
\begin{equation}\label{BuresHellRel}
   d_b(\rho,\sigma)\le d_h(\rho,\sigma)\le \sqrt{2}\,d_b(\rho,\sigma).
\end{equation}

In the following we show that  the Bures-Wasserstein and Hellinger metrics are equivalent in the cone of positive semidefinite matrices.
\begin{proposition}\label{BuresHellRelPos}
Let $A,B\in \mathcal{D}_n$. Then,
\begin{equation*}
   d_b(A,B)\le d_h(A,B)\le \sqrt{2}\,d_b(A,B).
\end{equation*}
\end{proposition}
\begin{proof}
The first inequality follows from the fact that
$$\Tr(A^{1/2}BA^{1/2})^{1/2}\ge \Tr(A^{1/2}B^{1/2}),$$
which is a consequence of the famous Araki-Lieb-Thirring inequality \cite{ArakiAude}. 

For the second inequality, let $\rho = A/\Tr(A)$ and $\sigma = B/\Tr(B)$. Since $\rho,\sigma \in \mathcal{D}_n^1$, from inequality \eqref{BuresHellRel} it implies that
\begin{align*}
   & &    d_h^2(\rho, \sigma) = 2-2 \, \Tr(\rho^{1/2}\sigma^{1/2}) \le 4-4\, \Tr((\rho^{1/2}\sigma \rho^{1/2})^{1/2})=2d_b^2(\rho,\sigma),
\end{align*}
or, 
$$
2\Tr((\rho^{1/2}\sigma \rho^{1/2})^{1/2}) \le 1+\Tr(\rho^{1/2}\sigma^{1/2}).
  $$
By substituting in $\rho = A/\Tr(A)$ and $\sigma = B/\Tr(B)$, from the last inequality and the AGM inequality we obtain that
\begin{align*}
2\, \Tr((A^{1/2}B A^{1/2})^{1/2}) & \le \Tr(A)^{1/2}\Tr(B)^{1/2} + \Tr(A^{1/2}B^{1/2})  \\ 
& \le \frac{1}{2}\Tr(A+B) + \Tr(A^{1/2}B^{1/2}).
\end{align*}
Consequently, 
\begin{align*}
   d_h^2(A,B) & = \Tr(A+B)-2\, \Tr(A^{1/2}B^{1/2}) \\
   & \le 2(\Tr(A+B)-2\, \Tr(A^{1/2}B A^{1/2})^{1/2}) \\
   & = 2d_b^2(A,B).
    \end{align*}
Thus, $$d_h(A,B) \le \sqrt{2} \, d_b(A,B).$$  
\end{proof}


Now we are ready to show that for the matrix power mean $\mu_p(t;A,B))$ the function $d_b(A,\mu_p(t;A,B))$ is monotonically increasing for $t\in [1/2, 1]$.
\begin{theorem}\label{InBetBur}
Let $A,B\in \mathcal{D}_n$, $1/2 \le p \le 1$ and $1/2 \le t \le 1$. Then,
$$d_b(A,\mu_p(t;A,B))\le d_b(A,B).$$
\end{theorem}
\begin{proof}Firstly, we show that for any positive positive semidefinite matrices $A$ and $B$, for $1/2 \le p \le 1$and  $0\le t \le 1$,
\begin{equation}\label{BuresHellRelPos1}
d_b(A,\mu_p(t;A,B))\le d_h(A,\mu_p(t;A,B))\le \sqrt{1-t} \, d_h(A,B).
\end{equation}
 Indeed, it is well-known \cite{BJL} that $$d_b(A,B)=\min_{U\in U(n)}\|A^{1/2}-B^{1/2}U\|_2,$$ where $U(n)$ is the group of unitary matrices of order $n$. In particular,
\begin{align*}d_b(A,\mu_p(t;A,B))^2 &\le \|A^{1/2}-\mu_p(t;A,B)^{1/2}\|^2_2 \\ & =\Tr\big(A+\mu_p(t;A,B)-2A^{1/2}\mu_p(t;A,B)^{1/2}\big).
\end{align*}
By the operator convexity of $x\mapsto x^{1/p}$, the operator concavity of $x\mapsto x^{1/2p}$, and the positivity of $A$, we obtain
\begin{align*}d_b(A,\mu_p(t;A,B))^2 &\le \Tr\big((1+t)A+(1-t)B-2(tA+(1-t)A^{1/2}B^{1/2})\big)\\
&=(1-t) \Tr\big(A-2A^{1/2}B^{1/2}+B)
\\ & = (1-t)\|A^{1/2}-B^{1/2}\|_2^2=(1-t)d_h(A,B)^2.\end{align*} 
From here, applying the square root function to both side we obtain (\ref{BuresHellRelPos1}).
Therefore, for $t \in [1/2, 1]$, we have 
\begin{align*}
   d_b(A, \mu_p(t;A,B)) & \le \sqrt{1-t} \, d_h(A,B)\le \frac{1}{\sqrt{2}} \, d_h(A,B) 
\end{align*}
Now, Proposition \eqref{BuresHellRelPos} implies that for $\frac{1}{2}\le t\le 1$,
$$d_b(A,\mu_p(t;A,B)) \le d_b(A,B).$$ 
By replacing $t$ with $\frac{1+t}{2}$, the previous inequality and  \cite[Lemma 1]{Audenaert2} imply that the real valued function $d_b(A,\mu_p(t;A,B))$ is monotonically increasing on $\frac{1}{2}\le t \le 1$. Thus, we prove the theorem. 
\end{proof}

In general, for $0\le p \le 1/2$ the ``in-betweenness" is not satisfied for $P_p(t;A,B).$ Consider, for example 
$$ A=\left(
\begin{array}{cc}
 5 & 14 \\
 14 & 41 \\
\end{array}
\right)\quad \text{ and } \quad  B=\left(
\begin{array}{cc}
 1 & -3 \\
 -3 & 18 \\
\end{array}
\right).$$
Then,
$$3.70465=d_b(A, P_{1/10}(1/8;A,B)) \not\le d_b(A,B)=3.60022.$$
Moreover, for
$$ A=\left(
\begin{array}{cc}
 167.621 & 47.0079 \\
 47.0079 & 14.0587 \\
\end{array}
\right) \quad \text{ and } \quad  B=\left(
\begin{array}{cc}
 37.903 & 23.3273 \\
 23.3273 & 14.4432 \\
\end{array}
\right),$$
we have
$$7.26351=d_h(A, P_{1/3}(1/4;A,B)) \not\le d_h(A,B)=7.22887.$$

From numerical simulations, it seems like  the ``in-betweenness" to be satisfied for $\mu_p(t;A,B)$ and  $p>0$ both for the Hellinger and Bures-Wasserstein metrics.




To finish this section in the following we show a similar result involving the quantum fidelity for $p\ge 1$ and both means $\mu_p(t;A,B)$ and $P_p(t;A,B)$.

\begin{proposition}
Let $A,B\in \mathcal{D}_n^1$, $p\ge 1$ and $0\le t \le 1$. Then,
$$ \mathcal{F}(A,\mu_p(t;A,B)) \ge \mathcal{F}(A,B)$$
and 
$$ \mathcal{F}(A,P_p(t;A,B)) \ge \mathcal{F}(A,B).$$
\end{proposition}
\begin{proof}
For $p=1$, this is easily shown to be true. Indeed,
\begin{multline*}
        \sqrt{\mathcal{F}(A,\mu_1(t;A,B))} =\Tr(t A^2+(1-t)A^{1/2}BA^{1/2})^{1/2}\\ \ge t+(1-t)\sqrt{\mathcal{F}(A,B)}\ge \sqrt{\mathcal{F}(A,B)},
\end{multline*}
where the first inequality follows from the operator concavity of the square root and the second is a consequence of Uhlman's theorem, that states that $\mathcal{F}(A,B)\in [0,1]$ with $\mathcal{F}(A,B)=1$ if and only if $A=B$.

Now, we note that if $0\le q \le p$ the function $x\mapsto x^{q/p}$ is operator concave, then
$$(tA^p+(1-t)B^p)^{q/p}\ge tA^q+(1-t)B^q.$$
In particular, when $q=1$, this implies
$$\sqrt{\mathcal{F}(A,\mu_p(t;A,B))}\ge \sqrt{\mathcal{F}(A,\mu_1(t;A,B))},$$
from which the result for $\mu_p(t;A,B)$ follows. The proof for $P_p(t;A,B)$ is similar as $P_1(t;A,B)=\mu_1(t;A,B)$ and 
$$(tI+(1-t)(A^{-1/2}BA^{-1/2})^p)^{q/p}\ge tI+(1-t)(A^{-1/2}BA^{-1/2})^q,$$
which implies 
$$\sqrt{\mathcal{F}(A,P_p(t;A,B))}\ge \sqrt{\mathcal{F}(A,P_1(t;A,B))}.\quad \qedhere$$ 
\end{proof}



\section{Monotonicity and in-sphere property in log-determinant metric}

In this section we consider monotonicity and  in-sphere property for matrix means in Log-Determinant metric. We discuss the connection between these two properties for the Heron mean. 

Firstly, note that for $a > 0$, the function $f(t) = a^{t/2}\nabla a^{-t/2}$ is increasingly monotone on $[0, 1]$, namely, for any $0\le t<s\le 1$,
\begin{equation}\label{lemma1}
    a^{t/2}\nabla a^{-t/2}\le a^{s/2}\nabla a^{-s/2}.
\end{equation}

\begin{theorem}\label{InBetLogDet}
For any two positive definite matrices $A$ and $B$, the function $t\mapsto d_l(A,A\sharp_t B)$ is increasingly monotone on $[0,1]$.
\end{theorem}
\begin{proof}
We want to show that $d_l(A,A\sharp_t B)\le d_l(A, A\sharp_s B)$ whenever $0\le t<s\le 1$. Firstly, note that for any positive matrix $X$ we have
 $$d_l(XAX,XBX)=d_l(A,B).$$
Consequently, for any Kubo-Ando mean $\sigma$ with the representing function $f_\sigma$ we have
 $$d_l(A,A\sigma B)= d_l(I,f_\sigma(C)),$$
where $C= A^{-1/2}BA^{-1/2}$. Therefore,  
$$d_l(A,A\sharp_t B)\le d_l(A, A\sharp_s B) \quad \hbox{if and only if} \quad d_l(I,C^t)\le d_l(I,C^s),$$
or 
$$\log \det C^{t/2}\nabla C^{-t/2}\le \log \det C^{s/2}\nabla C^{-s/2}.$$
The last inequality follows from the monotonicity of the logarithm and the determinant, and (\ref{lemma1}). 
\end{proof}

 Note that for positive definite matrices $A$ and $B$, $A\sharp_t B$ and $A\nabla_t B$ are the geodesic curve and the linear interpolation joining $A$ and $B$, respectively. If we consider those curves for $A\sharp B$ and $A\nabla B$, we have the following:  
$$A\diamondsuit_{t} B = (A\sharp B)\sharp_t(A\nabla B),  \quad t \in [0, 1],$$
and 
$$
H_t(A, B) = t A\sharp B + (1-t) A\nabla B, \quad t \in [0, 1].
$$ 
The geodesic curve $A\diamondsuit_t B$ between $A\sharp B$ and $A\nabla B$ was introduced and studied by the authors in \cite{DDF2}.   In the following we show that $d_l(A, A\diamondsuit_t B)$ is monotonic in $t\in [0, 1]$.

\begin{theorem}\label{thmoc}
For any two positive definite matrices $A$ and $B$ and $0\le t\le 1$,
$$d_l(A,A\diamondsuit_t B)\le d_l(A,B).$$
\end{theorem}
\begin{proof}
Firstly, we show that for any two positive definite matrices $A$ and $B$ and $0\le t\le 1$,
\begin{equation}\label{moc}
d_l(A,A\sharp B)+d_l(A\sharp B, A\nabla B)\le d_l(A,B).
\end{equation}
For that,  it suffices to show that
$$d_l(I,C^{1/2})+d_l(I, C^{1/2}\nabla C^{-1/2})\le d_l(I,C),$$
where $C= A^{-1/2}BA^{-1/2}$. The last inequality is equivalent to the following:
\begin{multline*}
    \log \det C^{1/4}\nabla C^{-1/4}+\log \det \frac{(C^{1/2}\nabla C^{-1/2})^{1/2}+(C^{1/2}\nabla C^{-1/2})^{-1/2}}{2} \\ \le \log \det C^{1/2}\nabla C^{-1/2},
\end{multline*}
or,
\begin{multline*}
     \det \left( C^{1/4}\nabla C^{-1/4} \cdot \frac{(C^{1/2}\nabla C^{-1/2})^{1/2}+(C^{1/2}\nabla C^{-1/2})^{-1/2}}{2}\right)  \\ \le \det  C^{1/2}\nabla C^{-1/2}.
\end{multline*}
By the operator concavity of the square root, we have
$$C^{1/4}\nabla C^{-1/4} \le \left(C^{1/2}\nabla C^{-1/2}\right)^{1/2}.$$
Therefore, the desired inequality (\ref{moc}) follows from
$$ \frac{X^{1/2}+ X^{-1/2}}{2} \le X^{1/2} \quad \hbox{with} \quad X = C^{1/2}\nabla  C^{-1/2}$$
which is obvious because $X \ge 1$.

Finally, on account of (\ref{moc}) and Theorem \ref{InBetLogDet} we have
\begin{align*}
d_l(A,A\diamondsuit_t B) & \le d_l(A,A\sharp B)+d_l(A\sharp B, A\diamondsuit_t B) \\
& \le d_l(A,A\sharp B)+d_l(A\sharp B, A\nabla B) \\
    & \le d_l(A,B).
\end{align*}
\end{proof}

\begin{remark}
Notice that, $\diamondsuit_t$ is a symmetric Kubo-Ando mean. Therefore, \cite[Theorem 8]{DDF} implies that $\diamondsuit_t$ does not satisfy the distance monotonicity or in-beteweenness with respect to any metric induced from a unitarily invariant norm. Moreover, while $\diamondsuit_t$ satisfies the in-betweenness with respect to the log-determinant metric, it does not satisfy the distance monotonicity.
\end{remark}

It is worth noting that in \cite{HK} we considered a ``naive" in-sphere property with respect to the arithmetic mean $A\nabla B$, i.e., sphere centered at $A\nabla B$. It was showed in \cite[Theorem 2.2]{HK} that for any symmetric mean  $\sigma$ and arbitrary unitarily invariant norm $||| \cdot |||$  on $\mathbb{M}_n$, if 
\begin{equation*}\label{5}
||| A\nabla B-  A\sigma B |||   \le \frac{1}{2} ||| A-B ||| 
\end{equation*}
holds whenever $A, B\in \mathcal{D}_n$, then $\sigma$ is the arithmetic mean.

In another picture, from Theorem \ref{InBetLogDet} one can see that the function 
$d_l(A\sharp B, A\diamondsuit_t B)$ is monotonic on $[0, 1]$. Therefore, 
$$ 
d_l(A\sharp B, A\diamondsuit_t B) \le d_l(A\sharp B, A\diamondsuit_1 B)=d_l(A\sharp B, A\nabla B).
$$
From the proof of Theorem \ref{thmoc} it is clear that 
$$ 
d_l(A\sharp B, A\nabla B) \le \frac{1}{2}d_l(A, B).
$$
Therefore, 
$$d_l(A\sharp B, A\diamondsuit_t B) \le \frac{1}{2}d_l(A, B).$$
The last inequality means that the curve  $A\diamondsuit_t B\ (t \in [0, 1])$ lies inside the sphere centered at $A\sharp B$ with the radius $d_l(A, B)/2.$ That means, when we change center of spheres, we may have different pictures of the in-sphere property. 

To finish this paper, we show that the Heron mean satisfies the in-sphere property with respect to the geometric mean. And hence, the curves $A \sharp_t B,$ $A\diamondsuit_t B$, $H_t(A, B)$ lie inside  the sphere centered at $A\sharp B$ with the radius $\frac{1}{2}d_l(A, B).$ 

\begin{theorem}
For positive definite matrices $A$ and $B$ and for any $t\in [0, 1]$, 
\begin{equation}\label{heron}
d_l(A\sharp B, H_t(A, B)) \le \frac{1}{2}d_l(A, B).
\end{equation}
\end{theorem}
\begin{proof}
We have Inequality (\ref{heron}) is equivalent to the following 
\begin{align*}
d_l(A\sharp B, H_t(A, B)) & = d_l(A\sharp B, tA\sharp B + (1-t)A\nabla B) \\
& = d_l(C^{1/2}, tC^{1/2}+(1-t)(1\nabla C^{1/2})) \\
& = d_l (I, t+(1-t)C^{1/2}\nabla C^{-1/2}) \\
&= \log \det \frac{(t+(1-t)C^{1/2}\nabla C^{-1/2})^{1/2}+ (t+(1-t)C^{1/2}\nabla C^{-1/2}}{2} \\
& \le \log \det C^{1/2}\nabla C^{-1/2}.
\end{align*}
The last inequality follows from the fact that for $x \ge 1$ the function $h(t)= (t+(1-t)x)^{1/2}-(t+(1-t)x)^{-1/2}$ has a non-positive derivative and hence, is decreasing. Therefore, $h(t)$ attains the maximum value $2x$ at $t=0$.  
\end{proof}

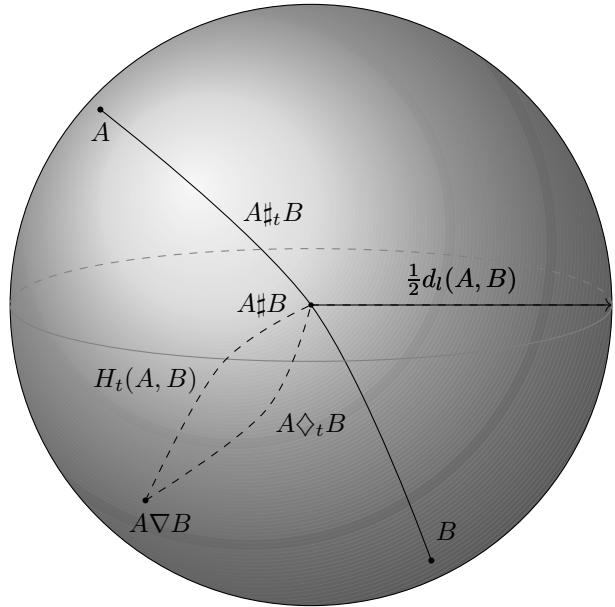
\begin{figure}
    \centering
  \begin{tikzpicture}
  \shade[ball color = gray!40!white, opacity = 0.6] (0,0) circle (4cm);
  \draw (0,0) circle (4cm);
  \fill[fill=black] (0,0) circle (1pt);
  \draw[->] (0,0) -- node[above]{$\frac{1}{2}d_l(A,B)$} (4,0);
  \node at (-0.65 ,0) {$A\sharp B$};
  \node at (-2.8 ,2.3) {$A$};
   \node at (1.8 ,-3.0) {$B$};
   \draw plot [smooth, tension=0.6] coordinates { (-2.8 ,2.6)  (0,0) (1.6 ,-3.4)};
   \node at (-0.5,1.2) {$A\sharp_t B$};
\draw[gray] (-4,0) arc (180:360:4 and 0.75);
  \draw[gray, dashed] (4,0) arc (0:180:4 and 0.75);
  \fill[black] (-2.8,2.6) circle (0.04cm);
  \fill[black] (1.6 ,-3.4) circle (0.04cm);
   \draw[dashed] (0,0 ) -- node[above]{$\frac{1}{2}d_l(A,B)$} (4,0);
    \node at (-2.0 ,-2.9) {$A\nabla B$};
    \draw [dashed] plot [smooth, tension=0.6] coordinates { (-2.2 ,-2.6) (-1.2,-0.8) (0,0)};
    \draw [dashed] plot [smooth, tension=0.6] coordinates { (-2.2 ,-2.6) (-.6,-1.4) (0,0)};
    \node at (-2.2,-1) {$H_t(A,B)$};
    \node at (0,-1.6) {$A\diamondsuit_t B$};
     \fill[black](-2.2 ,-2.6) circle (0.04cm);
\end{tikzpicture}
    \caption{Graphical depiction of the geodesic curve connecting $A$ and $B$, $A\sharp_t B$ and the sphere of radius $\frac{1}{2}d_l(A,B)$ that contains the curves $H_t(A,B)$ and $A\diamondsuit_t B$ connecting $A\sharp B$ and $A\nabla B$.}
    \label{fig:my_label}
\end{figure}


\begin{thebibliography}{1}
\bibitem{ArakiAude}
K. M.~R. Audenaert.
\newblock On the {A}raki-{L}ieb-{T}hirring inequality.
\newblock {\it Int. J. Inf. Sys. Sci.} \newblock 
4(1):78-83, 2008.


\bibitem{Audenaert2}
K. M. R. Audenaert.
\newblock In-betweenness, a geometrical monotonicity property for operator
  means.
\newblock {\em Linear Algebra Appl.} 438(4):1769-1778, 2013.
\newblock 16th \{ILAS\} Conference Proceedings, Pisa 2010.


\bibitem{BhatiaLim} R. Bhatia, T. Jain, Y. Lim.
\newblock On the Bures-Wasserstein distance between positive definite matrices.
\newblock {\it Expositiones Mathematicae}. DOI10.1016/j.exmath.2018.01.002

\bibitem{BJL} R. Bhatia, T. Jain, Y. Lim.
\newblock Inequalities for the Wasserstein mean of positive definite matrices, 	arXiv:1803.03357.



\bibitem{bhagwat} K.~V. Bhagwat, R.~Subramanian.
\newblock Inequalities between means of positive operators.
\newblock {\it Math. Proc. Camb. Phil.
  Soc.} 83(5):393-401, 1978.
  
\bibitem{Bhatia4}
R.~Bhatia, Y.~Lim, T. Yamazaki,
\newblock{Some norm inequalities for matrix means.}
\newblock {\it Linear Algebra Appl.} 501:112-122, 2016.

\bibitem{hoa07} T. H. Dinh,
\newblock {Some inequalities for the matrix Heron mean.}
\newblock {\it Linear Algebra Appl.}
528:321-330, 2017.

  \bibitem{DDF}
T. H. Dinh, R. Dumitru, J. A. Franco.
\newblock On the monotonicity of weighted power means for matrices.
\newblock {\em Linear Algebra Appl.} 527:128-140, 2017.

  \bibitem{DDF2}
T. H. Dinh, R. Dumitru, J. A. Franco.
\newblock Non-Linear Interpolation of the Harmonic-Geometric-Arithmetic Matrix Means.
\newblock Submitted, 2018.

\bibitem{HK} T. H. Dinh, B. K. T. Vo, T. Y. Tam. 
\newblock In-sphere property and reverse inequalities for matrix means. 
\newblock {\it To appear in Elect. J. Linear Algebra}. 2019.

\bibitem{Kubo-Ando} F.~Kubo, T.~Ando. Means of positive linear operators. {\it Math. Ann.} 246(3):205-224, 1980.


\bibitem{2016arXiv161103449S}
D.~{Spehner}, F.~{Illuminati}, M.~{Orszag}, W.~{Roga}.
\newblock {Geometric measures of quantum correlations with Bures and Hellinger distances}.
\newblock {\em ArXiv e-prints}, November 2016.



\bibitem{sra}
S. Sra.
\newblock A new metric on the manifold of kernel matrices with application to matrix geometric
means.
\newblock {\it In Advances in Neural Information Processing Systems (NIPS)}, December 2012.
\end{thebibliography}
\end{document}